\documentclass[12pt,psamsfonts,reqno]{amsart}

\usepackage{amssymb,amsfonts}
\usepackage[all,arc]{xy}
\usepackage{enumitem}
\setlist[enumerate]{topsep=2pt, itemsep=2pt}

\usepackage[top=1in, bottom=1.25in, left=1.25in, right=1.25in]{geometry}
\usepackage{amsfonts}
\usepackage{amsthm}
\usepackage{amsmath}
\allowdisplaybreaks

\usepackage{parskip}
\usepackage{mathrsfs}
\usepackage{graphicx}
\usepackage{amssymb}
\usepackage{xtab}
\usepackage{mathtools}

\usepackage{hyperref}
\usepackage{mathrsfs}
\usepackage{cleveref}

\usepackage{tikz} % https://q.uiver.app/ 
\usepackage{tikz-cd}
\usepackage{float}
\usetikzlibrary{positioning}
\usetikzlibrary{shapes.geometric,arrows}
\usepackage{mnotes}

\newtheorem{thm}{Theorem}[section]
\newtheorem{cor}[thm]{Corollary}
\newtheorem{prop}[thm]{Proposition}
\newtheorem{lem}[thm]{Lemma}

\newtheorem{quest}[thm]{Question}

\theoremstyle{definition}
\newtheorem{defn}[thm]{Definition}

\newtheorem{exmp}[thm]{Example}

\theoremstyle{remark}
\newtheorem{rem}[thm]{Remark}


\numberwithin{thm}{section}
\numberwithin{equation}{section}

\newcommand{\z}{\mathbb{Z}}
\newcommand{\q}{\mathbb{Q}}

\newcommand{\ad}{\mathbb{A}_k}

\DeclareMathOperator{\spec}{Spec}

\DeclareMathOperator{\Br}{Br}
\DeclareMathOperator{\br}{Br}
\DeclareMathOperator{\inv}{inv_v}

\DeclareMathOperator{\im}{im}
\DeclareMathOperator{\Ho}{H}
\DeclareMathOperator{\Sign}{Sign}
\DeclareMathOperator{\ev}{ev}
\DeclareMathOperator{\brv}{Br_{vert}}
\DeclareMathOperator{\res}{res}
\DeclareMathOperator{\Pic}{Pic}

\DeclareMathOperator{\Hom}{Hom}

\newcommand*{\myproofname}{Proof of \Cref{thm:ShaTrivial}}

\theoremstyle{definition}

\DeclareFontFamily{U}{wncy}{}
\DeclareFontShape{U}{wncy}{m}{n}{<->wncyr10}{}
\DeclareSymbolFont{mcy}{U}{wncy}{m}{n}
\DeclareMathSymbol{\Sh}{\mathord}{mcy}{"58} 
\DeclareMathSymbol{\Be}{\mathord}{mcy}{"42}

\title[Brauer groups of conic bundles over elliptic curves]{Brauer groups of conic bundles over elliptic curves}
\author[A. Alfaraj]{Abdulmuhsin Alfaraj}
\address{Abdulmuhsin Alfaraj\\
	Department of Mathematical Sciences \\
	University of Bath \\
	Claverton Down \\
	Bath \\
	BA2 7AY \\
	UK.}
\urladdr{}

\begin{document}

\begin{abstract}
	 We study the Brauer groups of regular conic bundles over elliptic curves defined over a number field $k$. We explicitly compute the Brauer group of the conic bundle when the singular fibres lie above $k$-points that are divisible by $2$ in $E(k)$, and the corresponding ramification fields are isomorphic. We apply the result to compute the Brauer group of a class of surfaces analogous to that of Ch\^{a}telet surfaces.  We investigate Brauer-Manin obstructions to weak approximation coming from the real places on such surfaces.
\end{abstract}
\maketitle
\tableofcontents
\setcounter{equation}{0}

\section{Introduction}
Let $C\subset \mathbb{A}^2(x,y)$ be a smooth curve over a number field $k$. One may ask the following question: how often is the $y$-coordinate (or the $x$-coordinate) of a $k$-point $(x,y)\in C(k)$ a sum of two squares? This is equivalent to asking how often does the fibre over $P\in C(k)$ along the conic bundle 
$$X:\quad x_0^2+x_1^2=yx_2^2 \quad \subset \mathbb{P}^2(x_0,x_1,x_2)\times \mathbb{A}^2(x,y)$$
have a $k$-point. When $C$ is the affine line, this is the same as asking how often is an element of $k$ a sum of two squares; over $\q$, we know by a famous theorem of Landau and Ramanujan that almost all integers are not sums of two squares (when ordered by their absolute value). By Faltings's theorem, the only remaining interesting case is when $C$ is an elliptic curve $E$ of positive rank. Bhakta, Loughran, Myerson, and Nakahara \cite{ECsieve} studied this question for an elliptic curve $E$ over $\q$ given by an integral Weierstrass
equation. They pose the following question: 
\begin{quest}
	Can the set $\{P\in E(\q)\> : \> X_P(\q)\neq \emptyset\}$ be infinite?
\end{quest} 
They proved that almost all multiples of a fixed $\q$-point in $E(\mathbb{R})^0$---the real connected component containing the identity---are not sums of two squares (when ordered by an appropriate height function). 

We build first steps to approach this problem via the study of the Brauer-Manin obstruction on conic bundles over elliptic curves. Let $X$ be a conic bundle over an elliptic curve defined over a number field $k$. Define $X(\ad)_\bullet$ to be $X(\ad)$ where we replace $X(k_v)$ in the restricted product with the set of connected components of $X(k_v)$, for all archimedean places $v$. One may ask the following question.
\begin{quest}
 	%Suppose that $X(k)\neq \emptyset$.
 	When is
	\begin{equation}\label{eqn:mainQ}
		\overline{X(k)} = X(\ad)_\bullet^{\Br}?
	\end{equation}
\end{quest}
If $X$ is a curve of genus $1$ and the Tate-Shafarevich group of the
Jacobian is finite, then \eqref{eqn:mainQ} holds (see \cite[Theorem 6.2.3, Corollary 6.2.4]{Sko01}). If $X$ is a Severi-Brauer scheme over an elliptic curve $E$ with finite Tate-Shafarevich group, then \eqref{eqn:mainQ} holds (see \cite[Proposition 14.3.9]{CTSK}). In particular, this is true for a regular conic bundle over $E$ with smooth fibres. If the conic bundle has singular fibres, then this is no longer the case. Indeed, there exists a conic bundle $X$ over an elliptic curve $E$ defined over a real quadratic field $k$ such that $X(\ad)_\bullet^{\Br}\neq \emptyset$ and $X(k)=\emptyset$ (see \cite[Theorem 14.3.7]{CTSK}). As a consequence of the results of this paper, we prove the following.

\begin{thm}\label{thm:intro}
	There exists an elliptic curve $E$ over $\q$ and a regular conic bundle $\pi:X\rightarrow E$ such that $X(k)\neq \emptyset$ and
	$$X(\ad)_\bullet^{\Br X} \neq X(\ad)_\bullet^{\pi^* (\Br E)}.$$
\end{thm}

The main objective of this paper is to study the Brauer groups of conic bundles over elliptic curves.  These surfaces can be seen as the natural next case to study after
conic bundles over the projective line. The Brauer group of conic bundles over curves are generally determined by the points in the base over which the fibres are singular. The problem then is reduced to finding Brauer elements on the curve that ramify at such points with prescribed residues. For the projective line, the Faddeev exact sequence allows us to determine such Brauer elements completely (see Remark \ref{rem:es1}). For curves of higher genus, the situation becomes more difficult to approach. Nevertheless, when the base is an elliptic curve $E$ over $k$, and the set of points over which the fibers are singular consists
of $k$-rational points that are divisible by $2$ in $E(k)$, and the corresponding ramification fields are isomorphic, we are able to give a complete description of the Brauer group of the conic bundle (see Theorem \ref{thm:CB-pts-in-2EQ}). We use this to fully describe the Brauer group for a class of examples of bundles over elliptic curves which resembles that of Ch\^{a}telet surfaces (see Example \ref{exmp:chat-surface}). 

\vspace{10pt}

\noindent\textbf{Acknowledgements.} I would like to thank my supervisor Daniel Loughran for his endless support. I would also like to thank Alexei Skorobogatov for the proof of Proposition \ref{prop:BrCminusP} and Jean-Louis Colliot-Thélène for useful discussions.

\section{Preliminaries}

\subsection{The vertical Brauer group}

\begin{defn}[{{\cite[11.1.1]{CTSK}}}]
	Let $f:X\rightarrow Y$ be a dominant morphism of schemes where $Y$ is integral with generic point $i:\eta\rightarrow Y$ and let $X_\eta$ be the generic fibre of $f$. Writing $j:X_\eta\rightarrow X$ for the natural inclusion, we have the following commutative diagrams
	\[\begin{tikzcd}
		{X_\eta} & X && {\Br (X_\eta)} & {\Br (X)} \\
		\eta & Y & {} & {\Br ( \eta)} & {\Br (Y)}
		\arrow["j", hook, from=1-1, to=1-2]
		\arrow["f", from=1-2, to=2-2]
		\arrow[from=1-1, to=2-1]
		\arrow["i", from=2-1, to=2-2]
		\arrow["{j^*}"', from=1-5, to=1-4]
		\arrow["{f^*}"', from=2-5, to=1-5]
		\arrow["{i^*}"', from=2-5, to=2-4]
		\arrow["\rho", from=2-4, to=1-4]
	\end{tikzcd}\]
	The \textit{vertical Brauer group} of $X/Y$ is 
	$$\Br_{\text{vert}} (X/Y):=\{ A\in \Br(X)\>:\>  j^*(A) \in \rho(\Br(\eta))\}.$$
\end{defn}

\begin{rem}
	If $X$ is regular and integral with generic point $\eta'$, then by \cite[Theorem 3.5.5]{CTSK} we have 
	$$\Br(X)\subset \Br(X_\eta)\subset \Br(\eta').$$
	Thus $\brv (X/Y)=\rho(\rho^{-1} (\br X))$, giving the exact sequence
	$$0\rightarrow \ker \rho \rightarrow \rho ^{-1}(\br(X))\rightarrow \brv(X/Y)\rightarrow 0.$$ 
\end{rem}

Let $f:X\rightarrow Y$ be a dominant flat morphism of smooth, proper, and geometrically integral varieties over a field $k$ of characteristic zero. Let $Y^{(1)}$ be the set of codimension $1$ points of $Y$ and write $X_P$ for the fibre of $X$ over $P\in Y^{(1)}$. For $P\in Y^{(1)}$, the fibre $X_P$ is a union of codimension $1$ points $V\subset X$ each of multiplicity $m_V$. The restriction map
$$  \text{res}_{k(V)/k(P)}: \Ho^1(k(P), \mathbb{Q}/\mathbb{Z}) \rightarrow  \Ho^1(k(V), \mathbb{Q}/\mathbb{Z})$$ 
induces via the morphism $X_P\rightarrow  P$ the map 
$$m_V \cdot \text{res}_{k(V)/k(P)}: \Ho^1(k(P), \mathbb{Q}/\mathbb{Z}) \rightarrow \bigoplus_{V\subset Y^(1)} \Ho^1(k(V), \mathbb{Q}/\mathbb{Z}).$$

\begin{prop}[{{\cite[p.266]{CTSK}}}]
	Let $f:X\rightarrow Y$ be a dominant flat morphism of smooth, proper, and geometrically integral varieties over a field $k$ of characteristic zero. Then we have a commutative diagram of exact rows
	\begin{equation}\label{diag:vertbr}
		\begin{tikzcd}
			0 & {\Br(X)} & {\Br(X_\eta)} & {\displaystyle\bigoplus_{P\in Y^{(1)}} \displaystyle\bigoplus_{V\subset X_P}\Ho^1(k(V), \mathbb{Q}/\mathbb{Z})} \\
			0 & {\Br(Y)} & {\Br(k(Y))} & {\displaystyle\bigoplus_{P\in Y^{(1)}} \Ho^1(k(P), \mathbb{Q}/\mathbb{Z})}
			\arrow[from=1-1, to=1-2]
			\arrow[from=1-2, to=1-3]
			\arrow["{\{\partial_V\}}", from=1-3, to=1-4]
			\arrow[from=2-1, to=2-2]
			\arrow["{f^*}", from=2-2, to=1-2]
			\arrow[from=2-2, to=2-3]
			\arrow[from=2-3, to=1-3]
			\arrow["{\{\partial_P\}}", from=2-3, to=2-4]
			\arrow[from=2-4, to=1-4]
		\end{tikzcd}
	\end{equation}
	where $\partial_P$ and $\partial_V$ denote the residue maps.  
\end{prop}

\begin{proof}
	The bottom row is exact by \cite[Theorem 3.7.3]{CTSK}. The top row is exact since $X$ is regular and $f$ is flat. Finally, The right part of the diagram commutes by the functoriality of residue.
\end{proof}

The following result is a weaker version of  \cite[Prop. 11.1.8]{CTSK}, where the base curve $\mathbb{P}^1$ is replaced by any integral smooth curve $C$. 

\begin{cor}\label{cor:VertBrOverC}
	Let $f:X\rightarrow C$ be a dominant morphism of smooth, proper, and geometrically integral varieties over a field $k$ of characteristic zero where $C$ is a curve. Suppose that $X_P$ is irreducible for all $P\in Y^{(1)}$ and let $K_P$ be the algebraic closure of $k(P)$ in $k(X_P)$. Then we have the following exact sequence
	$$0\rightarrow \br(Y) \rightarrow \rho^{-1}(\br(X))\rightarrow \displaystyle\bigoplus_{P\in Y^{(1)}} \ker(\res_{K_P/k(P)}). $$
\end{cor}

\begin{proof}
	Since $X$ is integral and $C$ is a regular integral curve, $f$ is flat \cite[Prop. III.9.7]{Har}. The result now follows by \eqref{diag:vertbr} and the definition of the vertical Brauer group.
\end{proof}

\begin{rem}\label{rem:es1}
	If the base $C$ is $\mathbb{P}^1$, then we have an extended exact sequence
	$$0\rightarrow \br(k) \rightarrow \rho^{-1}(\br(X))\rightarrow \displaystyle\bigoplus_{P\in Y^{(1)}} \ker(\text{res}_{K_P/k(P)}) \rightarrow \Ho^1(k, \mathbb{Q}/\mathbb{Z}), $$
	which follows by the Faddeev exact sequence.
\end{rem}

\subsection{The Brauer group of a curve minus a point}
Let $C$ be a smooth integral curve over a field $k$ of characteristic zero. Let $U$ be the complement  of finitely many closed point in $C$. Then, by \cite[Theorem 3.7.3]{CTSK}, we have an exact sequence 
$$0\rightarrow \br(C) \rightarrow \br(U)\rightarrow \displaystyle\bigoplus_{P\in C\setminus U} \Ho^1(k(P), \mathbb{Q}/\mathbb{Z}) . $$
In the case $C=\mathbb{P}^1$, then by the Faddeev exact sequence we obtain an extended exact sequence
$$0\rightarrow \br(C) \rightarrow \br(U)\rightarrow \displaystyle\bigoplus_{P\in C\setminus U} \Ho^1(k(P), \mathbb{Q}/\mathbb{Z}) \rightarrow \Ho^1(k, \mathbb{Q}/\mathbb{Z}) $$
which completely determines $\Br(U)$ in terms of the residues $\bigoplus_{P\in C\setminus U} \Ho^1(k(P), \mathbb{Q}/\mathbb{Z})$. In other words, given a prescribed set of residues at the points $P\in C\setminus U$, we can determine whether there is an element of $\br(U)$ realizing these residues. For a higher genus curve $C$, this is not the case. Nonetheless, when $C\setminus U$ is a $k$-point we show that the Brauer group remains the same.

The proof of the following result was provided by A. Skorobogatov via an e-mail correspondence.

\begin{prop}\label{prop:BrCminusP}
	Let $C$ be a smooth integral proper curve over a number field $k$. Let $P\in C(k)$ and set $U:=C\setminus P$. Then
	$$\frac{\Br(U)}{\Br_0(U)}=\frac{\Br(X)}{\Br_0(X)}.$$
\end{prop}
\begin{proof}
	First, by Tsen's theorem (see e.g. \cite[Theorem 1.2.12]{CTSK}) and \cite[Theorem 3.5.5]{CTSK}, we deduce that $\Br(U)=\Br_1(U)$ and $\Br(C)=\Br_1(C)$. Since  $\bar{k}[\bar{U}]^\times=\bar{k}^\times$ and $k$ is a number field, by \cite[Proposition 5.4.2, Remark 5.4.3]{CTSK}, we obtain the isomorphisms
	$$\br_1(U)/\Br_0(U)\cong \Ho^1(k, \Pic({U}^s)) $$ 
	and $$ \br_1(X)/\Br_0(X)\cong \Ho^1(k, \Pic({X}^s)).$$
	
	As $P\subset C(k)$, we have an exact sequence of $\Gamma_k$-modules
	$$0\rightarrow  \z  \rightarrow \Pic(\bar{C}) \rightarrow \Pic(\bar{U})\rightarrow 0$$
	which is right split by the degree map. Therefore, we have
	$$\Ho^1(k, \text{Pic}(\bar{U}))\cong \Ho^1(k, \text{Pic}(\bar{C})) \oplus \Ho^1(k, \z).$$
	Since $ \Ho^1(k, \z)=\Hom(\Gamma_k, \z)=0$, we deduce that 
	$$\br(U)/\br_0(U)\cong \Ho^1(k, \text{Pic}(\bar{U}))\cong \Ho^1(k, \text{Pic}(\bar{C}))\cong \br(C)/\br_0(C),$$
	as desired.
\end{proof}

\section{The Brauer group of a conic bundle over an elliptic curve}

\subsection{Conic bundles over curves}
Let $C$ be a smooth geometrically integral curve over a a field of characteristic zero and function field $k(C)$. Then any conic bundle over $C$ has a regular model where all the fibres are irreducible conics, which will allow us to restrict our study to regular conic bundles.

\begin{lem}[{{\cite[Lemma 11.3.2]{CTSK}}}]
	For any conic bundle $f:X\rightarrow C$ there exists a regular conic bundle $X'\rightarrow C$ such that $X'_{C}\cong X_C$ and and all the fibres of $X'\rightarrow C$ are reduced irreducible conics.
\end{lem}

\begin{lem}\label{lem:BrConicVert}
	Let $f:X\rightarrow C$ be a regular conic bundle. Then 
	$$\br(X)=\brv(X/C).$$
\end{lem}
\begin{proof}
	Let $K=k(C)$ and let $A\in\br (K)$ be the class of the quaternion algebra corresponding to the conic $X_K$ over $K$. If $A=0$, then $X_K\cong \mathbb{P}_K^1$ so that $\br (X_K) \cong \br(\mathbb{P}_K^1)\cong \Br(K)$. If $A\neq 0$, i.e. $X_K$ has no $K$-point, then by \cite[Proposition 7.2.1]{CTSK} we have the exact sequence
	\begin{equation}\label{seq:conicBr}
		0\rightarrow \mathbb{Z}/2\rightarrow \br(K)\xrightarrow{\rho} \br(X_K)\rightarrow 0
	\end{equation} 
	where $1\in \z /2$ is mapped to $A\in \br(K)$.
	In both cases, by the definition of the vertical Brauer group, we deduce that $\br(X)=\brv(X/C).$
\end{proof}

Let $f:X\rightarrow C$ be a regular conic bundle all of whose fibres are reduced irreducible conics. Since the generic fibre is smooth, by spreading out, we deduce that the set of points $P\in C$ over which the fibre is singular is finite; denote this set by $S$. Let $P\in S$. The restriction of $X$ to $\spec(\mathcal{O}_P)$ is isomorphic to the closed subscheme of $\mathbb{P}_{\mathcal{O}_P}^2$ given by $x^2-ay^2-\pi z^2=0$ where the image of $a\in\mathcal{O}_P^*$ in $k(P)$ is not a square, which we denote by $a_P$. The closed fibre $X_P\subset \mathbb{P}_{k(P)}^2$ is given by $x^2-a_Py^2=0$.  Set $F_P=k(P)(\sqrt{a_P})$. 
We have $A=(a,\pi)\in \br (K)$, so that the residue of $A$ at $P$ is $$\partial_P(A)=a_P\in k(P)^*/ {k(P)^*}^2=\Ho^1(k(P),\z / 2) .$$
Finally, since $F_P/k$ is a quadratic extension, we have $\Ho^1(F_P/k(P),\z/2)\cong \z/2$.

We have the following weaker version of \cite[Proposition 11.3.4]{CTSK} where the base now is any regular geometrically integral curve $C$ instead of $\mathbb{P}^1$. 

\begin{prop}\label{prop:Br-conic-C}
	Let $k$ be a field of characteristic zero. Let $C$ be a regular geometrically integral curve and let $f:X\rightarrow C$ be a regular conic bundle all of whose fibres are reduced irreducible conics. Let $S$ be the finite set of points $P\in C$ such that $X_P$ is not smooth over $k(P)$. If the class $A\in \br(k(C))$ associated to the conic $X_{k(C)}$ is not in the image of $\br(C)\rightarrow \br(k(C))$ (i.e. $A$ ramifies), then we have an exact sequence
	$$0\rightarrow \br(C) \rightarrow \br(X) \rightarrow \displaystyle\bigoplus_{P\in S} (\z/2)_P /\langle \partial(A) \rangle , $$ where $\partial(A)\in \displaystyle\bigoplus_{P\in S} (\z/2)_P = \displaystyle\bigoplus_{P\in S} \Ho^1(F_P/k(P),\z/2)$ is the vector with coordinates $\partial_P(A)$.
\end{prop}
\begin{proof}
	First, by Lemma \ref{lem:BrConicVert} we have that $\br(X)=\brv(X/C).$
	By Corollary \ref{cor:VertBrOverC}, we have the exact sequence 
	$$0\rightarrow \br(C) \rightarrow \rho^{-1}(\br(X))\rightarrow \displaystyle\bigoplus_{P\in S} \Ho^1(F_P/k(P),\z/2) ,$$
	by observing that $\ker(\text{res}_{F_P/k(P)}) = \Ho^1(F_P/k(P),\z/2)$. Since $A$ is not in the image of $\br(C)\rightarrow \br(k(C))$, we get an exact sequence 
	$$0\rightarrow \br(C) \rightarrow \rho^{-1}(\br(X))/\langle A \rangle\rightarrow \displaystyle\bigoplus_{P\in S} \Ho^1(F_P/k(P),\z/2)/\langle \partial(A) \rangle.$$
	By the exactness of \eqref{seq:conicBr}, we obtain the exact sequence
	$$0\rightarrow \mathbb{Z}/2\rightarrow \rho^{-1}(\br(X)) \xrightarrow{\rho} \brv(X/C)\rightarrow 0 $$ where $1\in \z /2$ is mapped to $A\in \br(K)$. Therefore, we have $$\rho^{-1}(\br(X))/\langle A \rangle \cong \brv(X/C)=\br(X),$$ which completes the proof.
\end{proof}

\begin{rem}\label{cor:BrX/E}
	Under the assumptions of the previous proposition, $\br(X)/\br(C)$ is a subgroup of $\bigoplus_{P\in S} (\z/2)_P /\langle \partial(A) \rangle$ and hence is finite. This also gives an upper bound on the size of $\Br(X)$ modulo $\Br(C)$ in terms of the number of the points over which the fibre is singular. 
\end{rem}

\subsection{Conic bundles over elliptic curves}
We now restrict to conic bundles over an elliptic curve $E$ over a number field $k$. By applying Proposition \ref{prop:Br-conic-C}, we provide a full description the Brauer group of the conic bundle  when the points over which the fibre is singular are divisible by $2$ in $E(k)$ with isomorphic ramification fields $F_P$. 

\begin{thm}\label{thm:CB-pts-in-2EQ}
	Let $E$ be an elliptic curve over a number field $k$. Let $f:X\rightarrow E $ be a regular conic bundle all of whose fibres are irreducible reduced conics. Let $S$ be the finite set of closed points of $E$ the fibres over which are singular. Suppose that $A\in \br (k(E))$ associated to $X_{k(E)}$ is not in the image of $\br(E)\rightarrow \br(k(E))$.
	If $S\subseteq 2E(k)$, and $F_P=F_Q$ for all $P, Q \in S$, then
	$$\br(X)/ \br(E)\cong (\z/2\z)^{|S|-2}.$$ 
	Moreover, fixing an embedding $E\hookrightarrow \mathbb{P}^2$ and a point $P_0\in S$, the generators have the form
	$$B_P=\left(a, \frac{\ell_P(X,Y,Z)}{\ell_{P_0}(X,Y,Z)}\right) \in \Br(k(E))$$ for $P\neq P_0$, where $\ell_P$ is the tangent line at any point $Q_p$ satisfying $-2Q_P=P$, and $X,Y,Z$ are the coordinates of $\mathbb{P}^2$.
\end{thm}
\begin{proof}
	Label the points in $S$ as $P_0, P_1, \ldots, P_{|S|-1}$ and let $a\in k$ be such that $\bar{a} \in k^*/ {k^*}^2$ corresponds to the quadratic extension $F_{P}$ for any $P\in S$ (which is independent of $P\in S$ by assumption). Fix an embedding $E\hookrightarrow \mathbb{P}^2$ and a point $P_0\in S$. Let $P\in S$. Since $P\in 2 E(k)$, there exists a point $Q_P\in E(k)$ satisfying $-2Q_P=P$, i.e., the tangent line at $Q_P$ intersects $E$ at $P$, which we denote by $\ell_P$. Therefore, we have  $$\text{div}(\ell_{P_i}/\ell_{P_0})=[P_i]+2[Q_{P_i}]-[P_0]-2[Q_{P_0}]$$ for all $i\in \{1,\ldots, |S|-1\}.$ Define the classes of the quaternion algebras
	$$B_{i}=\left(a, \frac{\ell_{P_i}(X,Y,Z)}{\ell_{P_0}(X,Y,Z)}\right) \in \Br(k(E)).$$
	By the residue formula \cite[(1.18) p.32]{CTSK}, for all $i$, we compute $$\partial_{P_i}(B_{i})=\overline{a}\in k^*/(k^*)^2, \quad \partial_{P_0}(B_{i})=\overline{a} \in k^*/(k^*)^2$$ and $\partial_{P}(B_{i})$ is trivial for all other closed points $P\in E$. Therefore, we deduce that ${B_i\in \rho^{-1}(\Br(X))}$ for all $i$.
	Hence, for all $i$, the image of $B_{i}$ along map $\partial$ in the exact sequence  
	$$0\rightarrow \Br(E)\rightarrow \rho^{-1}(\br(X)) \xrightarrow{\text{  }\partial \text{  }} \bigoplus_{i=0}^{|S|} \Ho^1(F_P/k(P),\z/2) = \bigoplus_{i=0}^{|S|} \z/2 $$ 
	is the vector having $1$ at the coordinates $0$ and $i$, and $0$ at all other coordinates. Therefore, $\{\partial (B_i)\}_{i=0}^{|S|-1}$ is a set of $|S|-1$ linearly independent {$\z/2$-vectors} which generate a {vector space} $V\subset \bigoplus_{i=0}^{|S|} \z/2$ of dimension $|S|-1$. 
	
	We claim that $\im(\partial)=V$. Indeed, for suppose that $V\subsetneq\im(\partial)$, i.e. $\im(\partial)=\bigoplus_{i=0}^{|S|} \z/2$. Then there exists $B\in \rho^{-1}(\Br(X))\setminus \Br(E)$ with residue of the form $(1,0,\ldots,0)\in \bigoplus_{i=0}^{|S|} \z/2$, i.e., $B$ ramifies only at the point $P_0$. Hence, $B\in \Br(E\setminus P_0)\setminus \Br(E)$, which contradicts Proposition \ref{prop:BrCminusP}. Therefore, by Proposition \ref{prop:Br-conic-C} we have
	$$\Br(X)/\Br(E) \cong \rho^{-1}(\Br(X))/(\langle A \rangle + \Br(E)) \cong V/\langle \partial(A) \rangle \cong  \bigoplus_{i=0}^{|S|-2} \z/2,$$
	where the last isomorphism follows since $\partial(A)\neq 0$. This completes the proof.
\end{proof}

\begin{exmp}[Sums of two squares on $E$]\label{exmp:conicbund-2tors}
	Let $E$ be the elliptic curve over a number field $k$ given by $$y^2=(x-a_1)(x-a_2)(x-a_3)$$ 
	where $a_1,a_2,a_3\in k$ and $a_1<a_2<a_3$. Set $P_i:=(a_i,0)$ and $P_0=O$.
	Consider the conic bundle $\pi: X \longrightarrow E$, given by the projectivization of $$x_0^2+x_1^2=yx_2^2 \quad \subset \mathbb{P}^2(x_0,x_1,x_2)\times \mathbb{A}^2(x,y).$$ 
	Then the fibre of a point $P=(x,y)\in E(k)$ has a $k$-point if and only if $y$ is a sum of two squares in $k$. 
	The singular fibres of lie above the points of $S=\{P_1, P_2, P_3, O\}$. Let $A \in \br(k(E))$ be the class associated to the generic fibre.
	For all $P\in S$, we have $\partial_P(A)=-1 \in \q^*/{\q^*}^2$ so that $$\partial(A)=(1,1,1,1) \in \displaystyle\bigoplus_{i=0}^3 (\z/2)_{P_i}.$$
	
	Suppose that $k=\q$. 
	We would like to find generators for $\br(X)/\br(E)$. The obvious quaternions $A_i=(-1,(X-a_iZ)/Z)$ for $i=1,2,3$ are unramified, hence belong to $\br(E)$. Let us suppose that $E$ has a non-trivial $4$-torsion point $Q$, so that $2Q=P_3$ for some $j\in \{1,2,3\}$, and let $\ell_{P_j}$ be the tangent line at $-Q$. Then, 
	$$B=\left(a, \frac{\ell_{P_j}(Z,Y,Z)}{Z}\right) \in \Br(k(E))$$
	ramifies only at $O$ and $P_j$ with residue equal to the class of $a$ in $\q^*/{\q^*}^2$. The image of $B$ in $\bigoplus_{i=0}^3 (\z/2)_{P_i}$ is the vector $(1,0,0,1)$. Then, we also have the element $A-B$ with residue $(0,1,1,0)$ that belongs to the same class as $B$ in $\Br(X)/\Br(E)\subset \bigoplus_{i=0}^3 (\z/2)_{P_i}/\langle \partial(A) \rangle$.  By Proposition \ref{prop:Br-conic-C}, we deduce that either $$\br(X)/\br(E)\cong \z/2\z \text{ or } \br(X)/\br(E)\cong \z/2\z\oplus \z/2\z,$$
	that is, there is at most one more generator. 
\end{exmp}

We now apply Theorem \ref{thm:CB-pts-in-2EQ} to compute the Brauer group of a class of surfaces analogous to that of Ch\^{a}telet surfaces.
\begin{cor}\label{exmp:chat-surface}
	Let $E\subset \mathbb{P}^2(X,Y,Z)$ be an elliptic curve over $k$ given by a Weierstrass equation. Let $\{P_1,\ldots, P_n\}\subset 2E(k)\setminus \{O\}$, for $n>0$, and let $t:= \#\{i : P_i \in  E(k)[2] \}$. Write $P_i=(x_i,y_i)\in E(k)$, for all $i$, and set $$f(x)=\prod_{i=1}^n (x-x_i).$$
	Let $X$ be a smooth compactification of the conic bundle
	$$X': \quad x_0^2-a x_1^2=f(x) x_2^2 \quad  \subset\mathbb{A}^2(x,y)\times \mathbb{P}^2(x_0,x_1,x_2),$$
	where $a\notin k^2$. Then 
	\[  \br(X)/ \br(E) \cong \left\{
	\begin{array}{ll}
		(\z/2\z)^{(2n-t-1)}, & \text{if $n$ is odd,} \\
		(\z/2\z)^{(2n-t-2)}, & \text{if $n$ is even.} \\
	\end{array} 
	\right. \]
\end{cor}

\begin{proof}
	First, since any two smooth compactifications of $X'$ are birationally equivalent, they will have isomorphic Brauer groups (see e.g. \cite[Corollary 6.2.11]{CTSK}). Thus, it suffices to compute the Brauer group for any smooth compactification. Let
	${\pi: X' \longrightarrow E\setminus O}$ be the projection map. Set $w:=1/x$ and let $X''\rightarrow E\backslash \{x=0\}$ be the conic bundle given by
	\[  \left\{
	\begin{array}{ll}
		y_0^2-ay_1^2=w^{n+1}f(1/w) y_2^2, & \text{if $n$ is odd,} \\
		y_0^2-ay_1^2=w^{n}f(1/w) y_2^2, & \text{if $n$ is even.} \\
	\end{array} 
	\right. \] 
	We define $X$ to be the gluing of $X'$ and $X''$ obtained by the relations 
	\begin{equation}\label{defn:Xgluing}
		w=1/x, \> y_i= w^{n+1} x_i \text{ when } n \text{ is odd and } w=1/x,\> y_i= w^{n} x_i \text{ when } n \text{ is even.}
	\end{equation} 
	Then $X$ is a smooth compactification for $X'$. Moreover, $X$ is a regular conic bundle all of whose fibres are irreducible reduced conics. Note that the fibre over $O\in E(k)$ is smooth when $n$ is even, and singular when $n$ is odd. Thus, the set of points $S$ in $E(k)$ over which the fibre is singular is
	\[  \left\{
	\begin{array}{ll}
		\{\pm P_i\}_{i=1}^n \cup \{O\}, & \text{if $n$ is odd,} \\
		\{\pm P_i\}_{i=1}^n, & \text{if $n$ is even.} \\
	\end{array} 
	\right. \] 
	with cardinalities $2n-t+1$ and $2n-t$, respectively. Let $A\in \br(\q(E))$ be class of the the quaternion algebra $\left(a,f(x)\right)$ corresponding to the generic fibre. By the residue formula \cite[(1.18) p.32]{CTSK}, we compute that  $\partial_{P_i}(A)\neq 0\in \z/2$ for $i$. Thus, $A$ ramifies, i.e., $A\notin \Br(E)$. Therefore, by Theorem \ref{thm:CB-pts-in-2EQ} the result follows.
\end{proof}

We now investigate the Brauer-Manin obstruction to weak approximation on this class of conic bundles.

\begin{prop}\label{prop:BMOtoWAonCoverE}
	Take the notation and assumptions in Theorem \ref{thm:CB-pts-in-2EQ}. 
	%Let $X$ be the smooth compactification defined in \eqref{defn:Xgluing}. 
	Suppose that $k=\q$, $a<0$, and $E(\mathbb{R})$ is connected. If $n>2$ and $S\cap E[2]=\emptyset$, then there exists $P,Q\in S$ and a Brauer group element 
	$$B=\left(a, \frac{\ell_{P}(x,y)}{\ell_{Q}(x,y)}\right) \in \Br(X)$$
	for some tangent lines $\ell_P$, $\ell_Q$ to $E$ intersecting $P$ and $Q$, respectively, such that  
	$$\inv_\infty\circ \ev_{B,\infty}: \> X(\mathbb{R}) \rightarrow  \z/2$$
	is surjective. In particular, we have
	$$X(\ad)_\bullet^{\pi^* (\Br E) + \langle B \rangle} \neq X(\ad)_\bullet^{\pi^* (\Br E)}.$$
\end{prop}

\begin{proof}
	First, order the points $P_i$ by the absolute value of their $x$-coordinate so that if $i>j$ then $x(P_i)>x(P_j)$. Also, we may assume that $y(P_i)>0$ for all $i$ by replacing $P_i$ with $-P_i$ if necessary. As each $P_i\in 2E(\q)$, there exists a line $\ell_i$ tangent to $E$ and intersecting $P_i$ at multiplicity $1$. It suffices to find two real points $R_1,R_2\in E(\mathbb{R})$ such that
	\begin{enumerate}
		\item $X_{R_i}(\mathbb{R})\neq \emptyset$ (equivalently $f(R_i)>0$) for $i=1,2$;
		\item and $\Sign(\ell_i(R_1)/\ell_j(R_1))\neq \Sign(\ell_i(R_2)/\ell_j(R_2)) $ for some $i,j\in \{1,\ldots,n\}$, since
	\end{enumerate}  
	\[ \inv_\infty \left( a, \frac{\ell_{i}(R_i)}{\ell_{j}(R_i)}\right) = \left\{
	\begin{array}{ll}
		0, & \text{ if } \ell_i(R_i)/\ell_j(R_i)>0 \\
		1, & \text{ if } \ell_i(R_i)/\ell_j(R_i)<0. \\
	\end{array} 
	\right. \]
	We claim that we can find such real points for $i=1$ and $j=3$. Indeed consider the following segments of $E(\mathbb{R})$:
	\[ \left\{
	\begin{array}{ll}
		E_1:=\{(x,y)\in E(\mathbb{R})\>:\> y>0,\> x<x(P_1) \} \\
		E_2:=\{(x,y)\in E(\mathbb{R})\>:\> y>0,\> x(P_2)>x>x(P_1) \} \\
		E'_2:=\{(x,y)\in E(\mathbb{R})\>:\> y>0,\> x(P_3)>x>x(P_2) \} \\
		E_3:=\{(x,y)\in E(\mathbb{R})\>:\> y>0,\> x>x(P_3) \} \\
	\end{array} 
	\right. \]
	Then either $f(x)>0$ for all $x\in E_1 \cup E'_2$ or $f(x)>0$ for all $x\in E_2 \cup E_3$. Moreover, one can verify that 
	\[ \left\{
	\begin{array}{ll}
		\Sign(\ell_1(x,y)/\ell_3(x,y)) \text{ is constant on } E_2\cup E'_2 \\
		\Sign(\ell_1(x,y)/\ell_3(x,y)) \text{ is constant on } E_1\cup E3 \\
		\Sign(\ell_1(x,y)/\ell_3(x,y)) \text{ attains both signs on } E_1\cup E_2\cup E'_2 \cup E_3
	\end{array} 
	\right. \]
	This shows the existence of the claimed real points $R_1$ and $R_2$, which completes the proof.
\end{proof}

\vspace{5pt}

\noindent\textbf{Proof of Theorem \ref{thm:intro}.} Let $E$ be an elliptic curve over $\q$ of positive rank such that $E(\mathbb{R})$ is connected. Let $\{P_1,\ldots, P_n\}\subset 2 E(\q)\setminus E[2]$ be a finite set of points where $n>2$. Then the conic bundle associated to $E$ and $\{P_1,\ldots, P_n\}$ in Example \ref{exmp:chat-surface} gives the desired example.

\end{document}